\documentclass [a4paper, 12t]{article}
\usepackage{amsmath}
\usepackage{amsfonts}
\usepackage{amssymb}
\usepackage{amsthm}
\usepackage{makeidx}
\usepackage{graphicx}
\usepackage{pb-diagram}
\usepackage{epstopdf}
\usepackage{fancyhdr}
\usepackage{enumitem}
\usepackage[all]{xy}

\newtheorem{teo}{Theorem}

\theoremstyle{definition}

\newtheorem{con*}{Conjecture}

\begin{document}
\setlength{\parskip}{1ex plus 0.5ex minus 0.2ex}
\begin{center}

\textbf{THE BANACH-ALAOGLU THEOREM IS EQUIVALENT TO THE TYCHONOFF THEOREM FOR COMPACT HAUSDORFF SPACES}\\
\end{center}
\begin{center}
\textsl{By STEFANO ROSSI}
\end{center}
$$$$
$$$$
\textbf{Abstract}\quad In this brief note we provide a simple approach to give a new proof of the well known fact that the Banach-Alaoglu theorem and the Tychonoff product theorem for compact \emph{Hausdorff}  spaces are equivalent.

$$$$
The Banach-Alaoglu theorem \cite{Ala} is one of the most useful compactness result available in Functional Analysis. It can be stated as follows:
\begin{teo}[Banach-Alaoglu]
Let $\mathfrak{X}$ be a normed space, then the unit ball $\mathfrak{X}^*_1$ of its dual space is weakly*-compact.
\end{teo}
As it is well known, the proof of the previous theorem crucially depends on Tychonoff's product theorem, since one can embeds $\mathfrak{X}^*_1$ in a suitable product of compact Hausdorff spaces, namely the product of closed disks in the complex plane. Nevertheless, if $\mathfrak{X}$ is a separable Banach space, the unit ball of its dual is metrizable in the weak* topology, so one can prove the theorem by showing that $\mathfrak{X}^*_1$ is \emph{sequentially} compact. This does not require Tycnohoff theorem, rather a diagonal argument does the job, as it was pointed out by Banach in his classic reference \cite{Ban}. Anyway, such a proof requires a countable version of the Axiom of Choice.\\
On the other hand, Tychonoff theorem is the most general compactness result available in point set topology. It can be stated in the following manner:
\begin{teo}[Tychonoff]
Let $\{X_{\alpha}:\alpha\in I\}$ be any family of compact spaces. Then the product $\prod_{\alpha\in I} X_{\alpha}$ is a compact space with respect the product topology.
\end{teo}
For the reader's convenience, we recall that the product topology on $\prod_{\alpha\in I} X_{\alpha}$  is the coarser topology for which the projection maps $p_{\lambda}:\prod_{\alpha\in I} X_{\alpha}\rightarrow X_{\lambda}$ ($p_{\lambda}(x)=x(\lambda)$ for each $x\in\prod_{\alpha\in I} X_{\alpha}$ ) are continuous.\\The proof of the previous theorem requires the Axiom of Choice in one of its equivalent formulations (for instance Zorn's lemma). Actually it is equivalent to the Axiom of Choice, as it has been shown in \cite{Kell} by Kelley in 1950.

What we are going to prove is that the Banach-Alaoglu theorem directly implies the following weaker version of the Tychonoff product theorem:
\begin{teo} [Tychonoff theorem: Hausdorff case]
\label{wtychonoff}
Let $\{K_{\alpha}\}$ be any family of compact \emph{Hausdorff} spaces. Then the product $\prod_{\alpha}K_{\alpha}$ is a compact Hausdorff space with respect to the product topology.
\end{teo}
Before giving our proof, one probably should mention that the equivalence we have to prove is already known in literature through a quite indirect argument. To be more precise, let us consider the following statements:
\begin{teo} [Prime ideal theorem for Boolean algebras: BPI]
Every Boolean algebra contains a prime ideal.
\end{teo}
\begin{teo} [Ultrafilters Lemma: UF] Let $\mathcal{F}$ be a filter on a set $X$. Then there exists an ultrafilter $\mathcal{G}$ such that $\mathcal{F}\preceq \mathcal{G}$.
\end{teo}
It is a very well known fact that BPI and UF are equivalent and they both are equivalent to the weaker version of Tychonoff theorem \ref{wtychonoff}. In \cite{Lux}, among other interesting things dealing with the Hanh-Banach extension theorem and its equivalents, it is shown that the Banach-Alaoglu theorem is equivalent to BPI. Clearly this shows that the Banach-Alaoglu theorem and the Tychonoff theorem \ref{wtychonoff} are equivalent.\\
Nevertheless, to what I know, direct proofs of this equivalence are not known. This paper is devoted to supply this lack, providing a very simple approach in a pure functional analytic context.

In what follows, we will assume the Zermelo-Fraenkel axioms (ZF) for set theory deprived of the Axiom of Choice (AC), which easily implies all the theorems quoted above.\\ As a consequence of ZF axioms\footnote{In particular, thanks to the infinite axiom.}, we can assume the existence of the field $\mathbb{R}$ of real numbers, together with its completeness property (existence of the least upper bound for superiorly bounded subset of $\mathbb{R}$). We observe that the construction of $\mathbb{R}$ (starting from the set of natural numbers $\mathbb{N}$) is independent of AC.\\
We will think $[0,1]=\{x\in\mathbb{R}:0\leq x\leq 1\}\subset\mathbb{R}$ endowed with its relative topology, the topology  on $\mathbb{R}$ being metrizable  with distance function given by $d(x,y)=|x-y|$ for each $x,y\in\mathbb{R}$). As a metrizable space, $[0,1]$ is Hausdorff. We will not need the fact that $[0,1]$ is compact (Heine-Borel theorem).\\
The first step is to prove that the Banach-Alaoglu theorem implies that Tychonoff cubes $[0,1]^{\Lambda}\doteq\{f:\Lambda\rightarrow [0,1]\}$ are compact with respect to the pointwise convergence topology (that does coincide with the product topology), $\Lambda$ being any set of indexes. Observe that $[0,1]^{\Lambda}$ is not empty (without assuming the AC), since surely one has $[0,1]\subset [0,1]^{\Lambda}$, by taking the constant functions.

Here is our first result:
 \begin{teo}
Let $\Lambda$ any set of indexes, then the Alaoglu theorem implies that the Tychonoff cube $[0,1]^{\Lambda}$ is compact.
\end{teo}
\begin{proof}
Put $K\doteq [0,1]^{\Lambda}=\{f:\Lambda\rightarrow [0,1]\}$ (with pointwise topology). Given $\alpha\in \Lambda$, let $p_{\alpha}$ be the corresponding projection (evaluation map), \emph{i.e} $p_{\alpha}(f)=f(\alpha)$ for each $f\in K$. Let $\mathfrak{X}\subset C(K)$ be the real vector space spanned by the evaluation maps. Note that the sup norm  is well defined on $\mathfrak{X}$, since the evaluation maps are bounded and $\mathfrak{X}$ is linearly generated by $p_{\alpha}$, as $\alpha$ runs $\Lambda$. By the Alaoglu theorem $\mathfrak{X}^*_1$ is weakly*-compact. Let us consider the natural inclusion map $$i:K\rightarrow\mathfrak{X}^*_1$$ given by $i(f)(p)=p(f)$ for each $p\in\mathfrak{X}$, $f\in K$.\\
This embedding is obviously injective, moreover, thanks to the definition of $\mathfrak{X}$, it is a homeomorphism  between $K$ and $i(K)$. In fact, it is continuous, since $K\ni f\rightarrow \langle i(f),p\rangle= p(f)\in\mathbb{R}$ is a continuous function for each $p\in\mathfrak{X}$ as a linear combination of real valued continuous function (this is the point where some elementary topological properties of $\mathbb{R}$ are invoked). Finally, it is an open map onto its image. If $U\subset K$ is an open subset, we have to prove that $i(U)$ is open in $i(K)$, \emph{i.e} $i(U)=i(K)\cap V$, where $V\subset\mathfrak{X}^*_1$ is a weakly* open subset. To this aim, it is enough to take $U$ running a base of open subset, so  we can assume $U=\bigcap_{j=1}^n p_{\lambda_j}^{-1}(A_j)$, where $A_j\subset [0,1]$ is a finite family of open subset. A straightforward computation shows that $i(U)=i(K)\bigcap\{\varphi\in\mathfrak{X}^*_1: \langle\varphi,p_{\lambda_j}\rangle \in A_j\,\,\textrm{for each}\,j=1,2\dots,n\}$, so we get the conclusion, the second set in the right side  being an open subset as a finite intersection of open sets.\\
To conclude the proof, it only remains to observe that $i(K)$ is closed in $\mathfrak{X}^*_1$, since one has
$$i(K)=\bigcap_{\alpha\in \Lambda} \{\varphi\in\mathfrak{X}^*: \varphi(p_{\alpha})\in [0,1]\}$$
So $i(K)$ is compact as a closed subset of a compact space.
\end{proof}
To get the full result, the following theorem from point set topology is needed. We  quote here its proof for completeness' sake. We also want to to underline that its proof does not depend on AC.
\begin{teo}
Let $K$ be a compact Hausdorff space. Then there exist a set $\Lambda$ and a map $i: K\rightarrow [0,1]^{\Lambda}$, which is an homeomorphism between $K$ and $i(K)$.
\end{teo}
\begin{proof}
As a compact Hausdorff space, $K$ is normal, so Urysohn lemma applies (Urysohn lemma is independent on AC, moreover its proof is by construction). Let $\{U_{\alpha}\}$ be a base of open subset. If $\overline{U_{\alpha}}\subset U_{\beta}$, then there exists a continuous function $f:K\rightarrow [0,1]$ such that $f(x)=0$ for each $x\in U_{\alpha}$ and $f(x)=1$ for each $x\in X\backslash U_{\beta}$. Let $\Lambda$ be the set of pairs $(\alpha, \beta)$ such that $\overline{U_{\alpha}}\subset U_{\beta}$, so if $\lambda\in\Lambda$ there is a continuous function $f_{\lambda}$ (whose values are in $[0,1]$) as above. Let us define the map $i:K\rightarrow[0,1]^{\Lambda}$ by $i(x)=  \{f_{\lambda}(x)\}_{\lambda\in\Lambda}$ for each $x\in K$. Observe that the possibility of defining this map does not rely on AC, since the correspondence $\Lambda \in\lambda\rightarrow f_{\lambda}$ has a rule, namely we associate to $\lambda\in \Lambda$ the function $f_{\lambda}$ which is built as in the proof of Urysohn lemma. \\
Let $F\subset K$ be a closed subset and $x\notin F$. Then there is $U_{\alpha}$ such that $x\in U_{\alpha}$ and $U_{\alpha}\in K\backslash F$. By normality there is a $U_{\beta}$ such that  $\overline{U_{\beta}}\subset U_{\alpha}$ and $x\in U_{\beta}$; this means that there is a function $f_{\lambda}$ such that $f_{\lambda}(F)=1$ and $f_{\lambda}(x)=0$, so $i$ is an injective map. Clearly $i$ is continuous. Moreover $i$ is relatively open. To see this, let $A\subset K$ be an open set and $x\in A$, then $F=X\backslash A$ is closed, so there is a $f_{\lambda}$ such that $f_{\lambda}(x)=0$ and $f_{\lambda}(F)=0$, so
$$i(x)\in \{\omega\in [0,1]^{\Lambda}: \omega(\lambda)<\frac{1}{2}\}\cap i(K)\subset\i(A)$$
This concludes the proof, since $x\in A$ is arbitrary.
\end{proof}
Finally, we can easily prove the implication in its full generality. Before performing the proof, we want to observe that,
 given a family of closed subset $C_{\alpha}\subset X_{\alpha}$, where $\{X_{\alpha}: \alpha\in I\}$ is any collection of topological spaces, the product set $C\doteq\prod_{\alpha\in I}C_{\alpha}$ is a closed\footnote{More generally, the fact that  $\overline{\prod_{\alpha\in I} C_{\alpha}}=\prod_{\alpha\in I}\overline{C_{\alpha}}$ is not trivial, requiring the Axiom of Choice. Actually, this statement is equivalent to the AC, see \cite{Sc}.} subset of  $X\doteq \prod_{\alpha\in I} X_{\alpha}$,  simply because one has $C=\bigcap_{\alpha\in I}p_{\alpha}^{-1}(C_{\alpha})$.

Now we are able to perform our proof of the quoted implication: 
\begin {teo}
\label {eq}
Banach-Alaoglu theorem implies Tychonoff theorem for compact \emph{Hausdorff} spaces.
\end {teo}
\begin {proof}
Let $\{K_{\alpha}:\alpha\in I\}$ be a family of compact Hausdorff spaces, then there is  a family of index-set $\Lambda_{\alpha}$ such that $K_{\alpha}\subset [0,1]^{\Lambda_{\alpha}}$ (as a closed subspace). Observe that $\prod_{\alpha\in I}[0,1]^{\Lambda_{\alpha}}\cong[0,1]^{\bigcup_{\alpha\in I} \Lambda_{\alpha}}$ (the union being disjoint) as topological spaces. Put $K=\prod_{\alpha\in I} K_{\alpha}\subset[0,1]^{\bigcup_{\alpha\in I} \Lambda_{\alpha}}$. $K$ is a closed subset as  product of closed sets ($K_{ \alpha}$ is closed, because it is compact and $[0,1]^{\Lambda_{\alpha}}$ is Hausdorff). Hence $K$ is compact as a closed subset of a compact space.
\end{proof}
Thanks to theorem \ref{eq}, we can also easily conclude that the Banach-Alaoglu theorem and the (apparently more general) Bourbaki-Alaoglu\footnote{Here is the statement of the general Bourbaki-Alaoglu theorem: let $E$ be a locally convex space. If $U\subset E$ is a circled convex neighborhood 0f $0\in E$, then its polar set $U^{\circ}\subset E^*$ is compact in the $\sigma(E^*,E)$-topology.} theorem are equivalent, since the proof of the last depends just on the weaker version of Tychonoff theorem.

$$$$
\textsl{DIP. MAT. CASTELNUOVO, UNIV. DI ROMA LA SAPIENZA, ROME, ITALY}\\
\emph{E-mail address:} \verb"s-rossi@mat.uniroma1.it"

\end{document}